\documentclass[12pt]{amsart}

\usepackage[
  textwidth=13.5cm,
  hcentering,
  top=1in,
  bottom=1in
]{geometry}

\usepackage[T1]{fontenc}
\usepackage{lmodern}
\usepackage{amsmath,amssymb,mathtools}
\usepackage{microtype}
\usepackage{setspace}

\usepackage[colorlinks=true,citecolor=blue,linkcolor=blue,urlcolor=blue]{hyperref}
\hypersetup{
  pdftitle={Schatten criteria for positive Toeplitz operators on pluriharmonic Fock spaces},
  pdfauthor={Sam Looi},
  pdfsubject={Schatten criteria for positive Toeplitz operators on pluriharmonic Fock spaces},
  pdfkeywords={pluriharmonic Fock space, generalized Fock space, positive Toeplitz operator, Schatten class, Schatten criterion, Schatten norm, symmetrically normed ideal, Carleson measure}
}
\newtheorem{theorem}{Theorem}[section]
\newtheorem{lemma}[theorem]{Lemma}
\newtheorem{proposition}[theorem]{Proposition}
\newtheorem{corollary}[theorem]{Corollary}
\theoremstyle{remark}

\newcommand{\F}{F}
\newcommand{\PH}{\mathrm{PH}}
\newcommand{\Sp}{\mathcal S}
\newcommand{\C}{\mathbb C}
\newcommand{\dV}{\,dV}
\newcommand{\norm}[1]{\left\|#1\right\|}
\newcommand{\abs}[1]{\left|#1\right|}
\newcommand{\inner}[2]{\left\langle #1,#2\right\rangle}

\title{Comparing positive Toeplitz operators on holomorphic and pluriharmonic Fock spaces}
\author{Sam Looi}
\address{California Institute of Technology, Pasadena, CA 91125, USA}
\email{s.looi9309@gmail.com}
\date{}


\begin{document}
\raggedbottom

\begin{abstract}
We prove sharp norm comparisons between positive Toeplitz operators associated with the same positive Borel measure, acting respectively on holomorphic functions and on sums of holomorphic and antiholomorphic functions, both equipped with the Gaussian-weighted $L^2$ norm. For every positive Borel measure $\mu$ on $\mathbb{C}^n$ and $0<p<\infty$, both operators are compact with summable $p$th powers of their singular values (i.e. Schatten) exactly when $z\mapsto\mu(B(z,r))$ belongs to $L^p(\mathbb{C}^n)$ for some fixed radius $r>0$. 

For other weights, this agreement can fail dramatically. We construct a weight $e^{-2\phi}$ and a finite positive measure $\mu$ on $\mathbb{C}$ for which the operator on holomorphic functions has the same singular values as with the Gaussian weight and the same measure $\mu$. These values have summable $p$th powers for every $p>0$, yet the quadratic form on sums of holomorphic and antiholomorphic functions is unbounded in weighted $L^2$. The example still satisfies $\Delta\phi=2$, as for the Gaussian choice $\phi(z)=|z|^2/2$.
\end{abstract}

\maketitle
\markboth{S. Looi}{Comparing positive Toeplitz operators on holomorphic and pluriharmonic Fock spaces}

\section{Introduction}
The holomorphic Fock space consists of entire functions that are square integrable with a Gaussian weight. For positive Toeplitz operators on this space, the singular values have summable $p$th powers exactly when the masses of balls of a fixed radius form an $L^p$ function of their centres. Jaguzovi\'c and Vujadinovi\'c conjectured the same characterization holds on pluriharmonic Fock space, whose functions are sums of holomorphic and antiholomorphic functions \cite[Introduction]{JaguzovicVujadinovic2026}. We prove that it does hold for every positive Borel measure and every $0<p<\infty$ and we obtain comparisons with optimal constants between the holomorphic and pluriharmonic Schatten norms.

For each weight, we define the squared direct sum norm as the sum of the squared weighted $L^2$ norms of the holomorphic and antiholomorphic components, with the ambiguity in their constant terms fixed as described below. For the Gaussian weight, this normalization requires the antiholomorphic component to vanish at the origin. The components are then orthogonal, so the direct sum norm agrees with the weighted $L^2$ norm of their sum. For more general weights, the normalized components need not be orthogonal, and cancellation can make the norm of their sum much smaller than the direct sum norm. As a result, boundedness of a Toeplitz form with respect to the direct sum norm need not imply boundedness with respect to the inherited weighted $L^2$ norm.

In Theorem~\ref{thm:inherited-failure-intro}, we construct a weight $e^{-2\phi}$ on $\mathbb{C}$ with $\Delta\phi=2$ and a finite positive measure $\mu$ for which $z\mapsto\mu(B(z,r))$ belongs to $L^p(\mathbb{C})$ for every $p>0$. The holomorphic Toeplitz operator has exactly the same singular values as the Gaussian Toeplitz operator associated with the same measure, and belongs to every Schatten class. The pluriharmonic Toeplitz operator defined using the direct sum norm also belongs to every Schatten class; yet the Toeplitz form is unbounded on the pluriharmonic space equipped with the weighted $L^2$ norm. Thus, even equality of the Laplacian with that of the Gaussian potential does not suffice to extend the holomorphic local mass criterion to this pluriharmonic space.

Fix $n\geq1$ and $\alpha>0$, and let $dV$ denote Lebesgue measure on $\C^n$. Let $V(E):=\int_E\dV$ and define the Gaussian probability measure
\begin{equation}\label{eq:gaussian-measure}
 d\mu_\alpha(z):=\frac{1}{(\pi\alpha)^n}
 e^{-\abs{z}^2/\alpha}\dV(z).
\end{equation}
The Fock space $\F^2_\alpha=\F^2_\alpha(\C^n)$ is the closed subspace of $L^2(\C^n,d\mu_\alpha)$ consisting of entire functions. A twice continuously differentiable function $f$ on $\C^n$ is pluriharmonic if
\[
 \frac{\partial^2f}{\partial z_j\,\partial\overline z_k}=0,
 \qquad 1\leq j,k\leq n.
\] This definition implies that $f = g+\overline h$ where $g,h$ are entire.
We write $\PH^2_\alpha=\PH^2_\alpha(\C^n)$ for the subspace of $L^2(\C^n,d\mu_\alpha)$ consisting of pluriharmonic functions; it is a closed subspace, as Lemma~\ref{lem:splitting} shows.

For a positive Borel measure $\mu$ on $\C^n$, we consider the Toeplitz form
\begin{equation}\label{eq:toeplitz-form-intro}
 (f,g)\mapsto
 \int_{\C^n}f(z)\overline{g(z)}
 e^{-\abs{z}^2/\alpha}\,d\mu(z)
\end{equation}
whenever the integral is absolutely convergent. This form is bounded on a Hilbert space $H$ of functions if it is absolutely convergent for all $f,g\in H$ and
\[
 \abs{\int_{\C^n}f(z)\overline{g(z)}
 e^{-\abs{z}^2/\alpha}\,d\mu(z)}
 \leq C\norm{f}_H\norm{g}_H.
\]
Its representing operator is bounded and positive. We denote it by $T_\mu$ on $\F^2_\alpha$ and $T_\mu^{\mathrm{ph}}$ on $\PH^2_\alpha$. Taking inner products linear in the first variable, we have
\begin{equation}\label{eq:representing-operator-intro}
 \inner{T_\mu f}{g}
 =\int_{\C^n}f(z)\overline{g(z)}
 e^{-\abs{z}^2/\alpha}\,d\mu(z),
\end{equation}
and the same identity with $T_\mu^{\mathrm{ph}}$ holds on $\PH^2_\alpha$.
For background on Fock spaces and Toeplitz operators, we refer to \cite{Zhu2012}. Isralowitz and Zhu proved the holomorphic criterion for the full range $0<p<\infty$ for the classical Fock space in one dimension \cite[Theorems~4.4 and~5.4]{IsralowitzZhu2010}. This was extended to generalized Fock spaces on $\C^n$ \cite[Theorems~2.7 and~3.2]{IsralowitzVirtanenWolf2015}.

For a compact operator $T$, let $s_1(T)\geq s_2(T)\geq\cdots$ be its singular values. For $0<p<\infty$, we write $T\in\Sp_p$ if
\[
 \norm{T}_{\Sp_p}^p
 =\sum_{k=1}^\infty s_k(T)^p<\infty.
\]
We set $\norm{T}_{\Sp_p}=+\infty$ when $T\notin\Sp_p$.

\begin{theorem}\label{thm:main}
Let $\mu$ be a positive Borel measure on $\C^n$, let $0<p<\infty$, and fix $r>0$. The following statements are equivalent:
\begin{enumerate}
\item The form in \eqref{eq:toeplitz-form-intro} is bounded on $\PH^2_\alpha$, and $T_\mu^{\mathrm{ph}}\in\Sp_p(\PH^2_\alpha)$.
\item The form in \eqref{eq:toeplitz-form-intro} is bounded on $\F^2_\alpha$, and $T_\mu\in\Sp_p(\F^2_\alpha)$.
\item The function $z\mapsto\mu(B(z,r))$ belongs to $L^p(\C^n,dV)$.
\end{enumerate}
The third condition is independent of $r>0$. Whenever these equivalent conditions hold, the following estimates also hold:
\begin{equation}\label{eq:quantitative-estimate}
 \norm{T_\mu}_{\Sp_p}^p
 \leq \norm{T_\mu^{\mathrm{ph}}}_{\Sp_p}^p
 \leq 2^{\max\{1,p\}}\norm{T_\mu}_{\Sp_p}^p.
\end{equation}
\end{theorem}
The constants $1$ and $2^{\max\{1,p\}}$ in \eqref{eq:quantitative-estimate} are optimal; see Proposition~\ref{prop:sharp-constants}.

The proof uses the orthogonal splitting
\[
 \PH^2_\alpha=\F^2_\alpha\oplus\overline{\F^2_{\alpha,0}},
 \qquad
 \F^2_{\alpha,0}=\{h\in\F^2_\alpha:h(0)=0\}.
\]
On the holomorphic summand, the Toeplitz form is exactly the form defining $T_\mu$, whereas on the antiholomorphic summand, its quadratic form is unchanged when $\overline h$ is replaced by $h$, so this part is also governed by $T_\mu$, with $h$ restricted to $\F^2_{\alpha,0}$. What remains is to understand the interaction between the two summands. Although they are orthogonal in the Fock norm, the Toeplitz form can contain mixed terms involving both parts.

Positivity allows us to control this interaction. In Proposition~\ref{prop:positive-block}, we show how the Schatten norms of the operators associated with the individual summands bound the Schatten norm of the full operator. That proposition applies to any finite orthogonal direct sum and covers the full range $0<p<\infty$. In our setting, it gives the desired comparison without requiring the mixed terms to vanish. This is what allows us the freedom not to make any radiality assumption on $\mu$.

Along the way, we also prove bounds for sums of the largest singular values, which then allow us to compare the two operators in every symmetrically normed ideal, where the singular values are measured by more general norms than the power sums defining the Schatten norms. To state this extension, let $\Phi$ be a symmetric norming function: a norm on finitely supported sequences that is unchanged by permuting the entries or replacing them by their absolute values, and is normalized by $\Phi(1,0,\ldots)=1$.

For a compact operator $T$, set
\begin{equation}\label{eq:phi-norm}
 \norm{T}_\Phi
 =\sup_{N\geq1}
 \Phi\bigl(s_1(T),\ldots,s_N(T),0,\ldots\bigr).
\end{equation}
We use the convention that the associated symmetrically normed ideal $\mathcal I_\Phi$ consists of the compact operators for which \eqref{eq:phi-norm} is finite. The choices $\Phi_p(x)=(\sum_k\abs{x_k}^p)^{1/p}$, $1\leq p<\infty$, give the Schatten norms, including the trace and Hilbert--Schmidt norms. The maximum norm gives the operator norm on compact operators. The quasi-norms for $0<p<1$ are not symmetric norming functions; see~\cite{Simon2005} for the general theory.

\begin{corollary}[Symmetrically normed ideals]
\label{cor:symmetric-ideals}
Let $\mu$ be a positive Borel measure on $\C^n$. The holomorphic Toeplitz form is bounded and $T_\mu\in\mathcal I_\Phi$ if and only if the pluriharmonic Toeplitz form is bounded and $T_\mu^{\mathrm{ph}}\in\mathcal I_\Phi$. Whenever these conditions hold,
\begin{equation}\label{eq:symmetric-norm-estimate}
 \norm{T_\mu}_\Phi
 \leq\norm{T_\mu^{\mathrm{ph}}}_\Phi
 \leq2\norm{T_\mu}_\Phi.
\end{equation}
More precisely, for every $N\geq1$,
\begin{equation}\label{eq:ky-fan-comparison}
 \sum_{k=1}^N s_k(T_\mu^{\mathrm{ph}})
 \leq2\sum_{k=1}^N s_k(T_\mu).
\end{equation}
Both constants in \eqref{eq:symmetric-norm-estimate} are optimal.
\end{corollary}

For generalized weights, the preceding argument applies to the direct sum norm. To compare it with the inherited norm, let us define both spaces. 

Let $\phi\in C^2(\C^n;\mathbb R)$. We assume that its complex Hessian is bounded above and below by fixed positive multiples of the identity. Equivalently, $\phi$ satisfies the uniform Levi bounds
\[
 c\omega_0\leq dd^c\phi\leq C\omega_0,
 \qquad \omega_0=dd^c\abs{z}^2,
 \qquad c,C>0.
\]
Let $\F_\phi^2$ be the space of entire functions in $L^2(\C^n,e^{-2\phi}dV)$, with norm $\norm{\cdot}_\phi$.

We use the following notation:
\[
 \mathcal C_\phi=\F_\phi^2\cap\C,
 \qquad
 \F_{\phi,\circ}^2=\F_\phi^2\cap(\mathcal C_\phi)^\perp.
\]The first expression means $\mathcal{C}_\phi
 :=\{f\in F_\phi^2:f\text{ is constant}\}.$
The canonical pluriharmonic direct sum is
\[
 \PH_{\phi,\oplus}^2
 =\{g+\overline h:g\in\F_\phi^2,\ h\in\F_{\phi,\circ}^2\},
 \qquad
 \norm{g+\overline h}_{\phi,\oplus}^2
 =\norm{g}_\phi^2+\norm{h}_\phi^2.
\]
Removing $\mathcal C_\phi$ from the second summand makes this representation unique. We also consider the space
\[
 \PH_{\phi,L^2}^2
 =\{f\in L^2(\C^n,e^{-2\phi}dV):f\text{ is pluriharmonic}\}
\]
with its inherited weighted $L^2$ norm.
For a positive Borel measure $\mu$, let $T_{\mu,\phi}$ and $T_{\mu,\phi}^{\mathrm{ph},\oplus}$ denote the operators representing
\[
 q_{\mu,\phi}(f,u)
=\int_{\C^n}f(z)\overline{u(z)}e^{-2\phi(z)}\,d\mu(z)
\]
on $\F_\phi^2$ and $\PH_{\phi,\oplus}^2$, respectively, whenever the corresponding forms are bounded.

\begin{theorem}[Generalized weights]\label{thm:generalized-intro}
Let $\phi$ satisfy the Levi bounds above, and let $\mu$ be a positive Borel measure on $\C^n$. For every $0<p<\infty$ and every $r>0$,
\[
 T_{\mu,\phi}^{\mathrm{ph},\oplus}\in\Sp_p
 \quad\Longleftrightarrow\quad
 T_{\mu,\phi}\in\Sp_p
 \quad\Longleftrightarrow\quad
 \mu(B(\cdot,r))\in L^p(\C^n,dV),
\]
where membership assertions include boundedness of the defining forms. Whenever these conditions hold,
\[
 \norm{T_{\mu,\phi}}_{\Sp_p}^p
 \leq\norm{T_{\mu,\phi}^{\mathrm{ph},\oplus}}_{\Sp_p}^p
 \leq2^{\max\{1,p\}}\norm{T_{\mu,\phi}}_{\Sp_p}^p.
\]
Membership is also equivalent in every symmetrically normed ideal, with
\[
 \norm{T_{\mu,\phi}}_\Phi
 \leq\norm{T_{\mu,\phi}^{\mathrm{ph},\oplus}}_\Phi
 \leq2\norm{T_{\mu,\phi}}_\Phi.
\]
The constants are optimal uniformly over this class of weights.

If $\phi(e^{i\theta}z)=\phi(z)$ for all $z\in\C^n$ and $\theta\in\mathbb R$, then $\PH_{\phi,\oplus}^2=\PH_{\phi,L^2}^2$ with equality of norms, and the same conclusions hold on the inherited space. More generally, the local mass Schatten criterion holds on $\PH_{\phi,L^2}^2$ whenever
\[
 m e^{-2\psi(z)}\leq e^{-2\phi(z)}\leq M e^{-2\psi(z)}
 \qquad (z\in\C^n)
\]
for some $0<m\leq M<\infty$ and a weight $\psi$ satisfying the same type of Levi bounds and $\psi(e^{i\theta}z)=\psi(z)$.
\end{theorem}

Theorem~\ref{thm:generalized-direct-sum} proves the direct sum assertions using the holomorphic criterion~\cite[Theorems~2.7 and~3.2]{IsralowitzVirtanenWolf2015} and the positive block principle. Proposition~\ref{prop:circle-invariant-weight} and Corollary~\ref{cor:comparable-circle-weight} give the inherited norm assertions, and Proposition~\ref{prop:generalized-trace} gives the direct sum trace formula. After the first arXiv version of this paper appeared, J.~Virtanen suggested considering extensions to generalized Fock weights.

In several variables, scalar circle invariance allows nonradial Hermitian Gaussian weights; the comparison condition also allows bounded nonradial perturbations satisfying the Levi bounds. Without a condition relating the two norms, the local mass criterion can fail as follows.

\begin{theorem}
\label{thm:inherited-failure-intro}
There are an entire function $q$ on $\C$ and a finite positive discrete measure $\mu$ such that the weight
\[
 \phi(z)=\frac{\abs{z}^2}{2}+\operatorname{Re}q(z)
\]
satisfies
\[
 dd^c\phi=\frac12dd^c\abs{z}^2,
 \qquad
 \mu(B(\cdot,r))\in L^p(\C,dV)
 \quad (0<p<\infty,\ r>0),
\]
and $T_{\mu,\phi}\in\Sp_p$ for every $0<p<\infty$, while the form $q_{\mu,\phi}$ is unbounded on $\PH_{\phi,L^2}^2$.
\end{theorem}

We choose $q$ and holomorphic functions $g_j$ so that the two terms of $g_j+\overline{g_j}$ nearly cancel in weighted $L^2$, while their sum has a weighted value bounded away from zero at $a_j$. We put $\mu=\sum_j2^{-j}\delta_{a_j}$ and make the squared norms decay faster than $2^{-j}$. The holomorphic assertion follows from the unitary map $g\mapsto e^{-q}g$ to Gaussian Fock space.

Jaguzovi\'c and Vujadinovi\'c proved a trace class characterization~\cite[Proposition~3.3]{JaguzovicVujadinovic2026} and for $p>1$ obtained the necessary local mass condition by holomorphic compression~\cite[Proposition~3.5]{JaguzovicVujadinovic2026}. They also proved sufficiency for positive regular radial measures, for which the mixed block vanishes~\cite[Section~3.3]{JaguzovicVujadinovic2026}. Thus Theorem~\ref{thm:main} resolves their conjecture for $n\geq2$. In one dimension, harmonic and pluriharmonic functions coincide, and the criterion for $p\geq1$ is stated in~\cite[Theorem~4.2]{GouHuHuang2024}. Our theorem also covers $0<p<1$ in every dimension.

For pluriharmonic Bergman spaces, Choi studied positive Toeplitz operators \cite{Choi2007}; Schatten and Schatten--Herz criteria were obtained by Choi and Na \cite{ChoiNa2007} and by Na \cite{Na2009}. Fulsche developed the orthogonal direct sum and block matrix approach for Toeplitz operators on pluriharmonic Bergman and Fock spaces \cite[Sections~2 and~3.1, equation~(7)]{Fulsche2019}. Engli\v{s} studied the Berezin transform on the harmonic, rather than pluriharmonic, Fock space over $\C^n$ \cite{Englis2010}. In the plane, Vujadinovi\'c studied atomic decompositions and Carleson measures \cite{Vujadinovic2020,Vujadinovic2021}.

In Sections~\ref{sec:splitting} and~\ref{sec:forms} we prove the orthogonal splitting, specify the normalization of the reproducing kernel, and recall the holomorphic Schatten criterion. We prove estimates for positive block operators in Section~\ref{sec:positive-block}, then apply them in Section~\ref{sec:application} to prove Theorem~\ref{thm:main} and Corollary~\ref{cor:symmetric-ideals}. Section~\ref{sec:sharpness} treats traces and sharpness, and Section~\ref{sec:generalized-weights} treats generalized weights and the inherited norm.

\section{The orthogonal splitting}\label{sec:splitting}
With the convention that the inner product is linear in the first variable,
\[
 \inner{f}{g}_{2,\alpha}
 =\frac{1}{(\pi\alpha)^n}
 \int_{\C^n}f(z)\overline{g(z)}e^{-\abs{z}^2/\alpha}\dV(z).
\]
We use the convention $\mathbb N=\{0,1,2,\ldots\}$.

For a multi-index $m=(m_1,\ldots,m_n)\in\mathbb N^n$, we put $m!=m_1!\cdots m_n!$, $\abs{m}=m_1+\cdots+m_n$, and $z^m=z_1^{m_1}\cdots z_n^{m_n}$. The standard Gaussian calculation and the monomial expansion of entire functions show that
\[
 \inner{z^m}{z^\ell}_{2,\alpha}
 =\delta_{m\ell}\alpha^{\abs{m}}m!.
\]
Thus the functions
\[
 e_m(z)=\frac{z^m}{\alpha^{\abs{m}/2}\sqrt{m!}}
\]
form an orthonormal basis of $\F^2_\alpha$. Hence the reproducing kernel at $z$ is
\begin{equation}
 K_z(w)=\sum_{m\in\mathbb N^n}e_m(w)\overline{e_m(z)}
 =\exp\left(\frac{w\cdot\overline z}{\alpha}\right);
 \label{eq:holomorphic-kernel}
\end{equation}
in particular, $f(z)=\inner{f}{K_z}_{2,\alpha}$. These are the standard Fock space normalization and kernel formulas; see \cite{Zhu2012}.

\begin{lemma}[Orthogonal splitting]\label{lem:splitting}
Every $f\in\PH^2_\alpha$ has a unique representation
\[
 f=g+\overline h,
 \qquad g\in\F^2_\alpha,
 \qquad h\in\F^2_{\alpha,0}.
\]
The norms satisfy
\begin{equation}\label{eq:norm-splitting}
 \norm{f}_{2,\alpha}^2
 =\norm{g}_{2,\alpha}^2+\norm{h}_{2,\alpha}^2.
\end{equation}
As a result,
\begin{equation}\label{eq:orthogonal-splitting}
 \PH^2_\alpha
 =\F^2_\alpha\oplus\overline{\F^2_{\alpha,0}}.
\end{equation}
\end{lemma}

\begin{proof}
Pluriharmonicity implies that the differential form
\begin{equation*}
 \omega=\sum_{j=1}^n\frac{\partial f}{\partial z_j}\,dz_j
\end{equation*}
is holomorphic, and equality of mixed derivatives shows that it is closed. Since $\C^n$ is simply connected, there is an entire function $g$ with $dg=\omega$. Thus $f-g=\overline h$ for an entire function $h$. Adding $\overline{h(0)}$ to $g$ and subtracting $h(0)$ from $h$, we arrange that $h(0)=0$. If $g_1+\overline{h_1}=g_2+\overline{h_2}$, then $g_1-g_2=\overline{h_2-h_1}$ is constant. It is zero when $h_1(0)=h_2(0)=0$, proving uniqueness.

Since $gh$ is entire and $(gh)(0)=0$, its spherical mean vanishes at every radius. Polar integration therefore gives
\begin{equation*}
 \int_{B(0,R)}g(z)h(z)e^{-\abs{z}^2/\alpha}\dV(z)=0,
 \qquad R>0.
\end{equation*}
Expanding $\abs{g+\overline h}^2$ yields
\begin{equation*}
 \int_{B(0,R)}\abs{f(z)}^2\,d\mu_\alpha(z)
 =\int_{B(0,R)}\bigl(\abs{g(z)}^2+\abs{h(z)}^2\bigr)\,d\mu_\alpha(z).
\end{equation*}
Letting $R\to\infty$ and applying monotone convergence, we obtain~\eqref{eq:norm-splitting} and hence $g,h\in\F^2_\alpha$. For any $g\in\F^2_\alpha$ and $h\in\F^2_{\alpha,0}$, Cauchy-Schwarz gives $gh\in L^1(\C^n,d\mu_\alpha)$. Since the integral of $gh$ over each ball $B(0,R)$ vanishes, we may pass to the limit to obtain
\[
 \inner{g}{\overline h}_{2,\alpha}
 =\int_{\C^n}g(z)h(z)\,d\mu_\alpha(z)=0.
\]
\end{proof}

Since $K_0=1$, evaluation at the origin is continuous on $\F^2_\alpha$, and $\F^2_{\alpha,0}$ is closed. Conjugation is antiunitary, so both summands in~\eqref{eq:orthogonal-splitting} are closed. Hence $\PH^2_\alpha$ is closed in $L^2(d\mu_\alpha)$.

The reproducing kernel of $\F^2_{\alpha,0}$ is $K_z-1$. Hence the kernel at $z$ for $\overline{\F^2_{\alpha,0}}$ is $\overline{K_z-1}$. Adding the kernels of the two orthogonal summands gives
\begin{equation}\label{eq:ph-kernel}
 K_z^{\mathrm{ph}}(w)
 =K_z(w)+\overline{K_z(w)}-1
 =e^{w\cdot\overline z/\alpha}
  +e^{\overline w\cdot z/\alpha}-1.
\end{equation}
The splitting and its associated block form also appear in~\cite[Sections~2 and~3.1, equation~(7)]{Fulsche2019}.

\section{Toeplitz forms and kernel operators}\label{sec:forms}
The form in \eqref{eq:toeplitz-form-intro} is positive; if it is bounded on $H$, the Riesz representation theorem gives a unique bounded positive operator $S$ on $H$ such that
\begin{equation}\label{eq:representing-operator}
 \inner{Sf}{g}_H
 =\int_{\C^n}f(z)\overline{g(z)}
 e^{-\abs{z}^2/\alpha}\,d\mu(z).
\end{equation}
This definition follows the approach of \cite[Sections~2 and~4]{RozenblumVasilevski2014}, where Toeplitz operators with singular symbols are defined through sesquilinear forms. We next explain how it relates to the usual definition in terms of the reproducing kernel.

Let $H$ be either $\F^2_\alpha$ or $\PH^2_\alpha$, write $K_z^H$ for its reproducing kernel and let
\[
 \mathcal D_H=\operatorname{span}\{K_z^H:z\in\C^n\}.
\]
This space is dense in $H$, since a vector orthogonal to every $K_z^H$ vanishes identically.

\begin{lemma}[Agreement with the kernel definition]
\label{lem:kernel-form-agreement}
Suppose that, for every $f\in\mathcal D_H$, the formula
\begin{equation}\label{eq:kernel-operator}
 (T_{\mu,0}^Hf)(z)
 =\int_{\C^n}f(w)\overline{K_z^H(w)}
 e^{-\abs{w}^2/\alpha}\,d\mu(w)
\end{equation}
converges and defines a function in $H$. 
If $T_{\mu,0}^H$ extends to a bounded operator $T$ on $H$, then the form in~\eqref{eq:toeplitz-form-intro} is bounded on $H$, with representing operator $T$. Conversely, if the form is bounded on $H$, its representing operator is given pointwise by the kernel formula~\eqref{eq:kernel-operator}.
\end{lemma}

\begin{proof}
For $f,g\in\mathcal D_H$, the reproducing property gives
\begin{equation}\label{eq:kernel-form-identity}
 \inner{T_{\mu,0}^Hf}{g}_H
 =\int_{\C^n}f(w)\overline{g(w)}e^{-\abs{w}^2/\alpha}\,d\mu(w).
\end{equation}
Suppose $T_{\mu,0}^H$ extends to a bounded operator $T$. Given $f\in H$, let $f_j\in\mathcal D_H$ be a sequence converging to $f$ in $H$. By Fatou's lemma,
\begin{equation*}
 \int_{\C^n}\abs{f(w)}^2e^{-\abs{w}^2/\alpha}\,d\mu(w)
 \leq\liminf_j\inner{Tf_j}{f_j}_H
 \leq\norm{T}\norm{f}_H^2.
\end{equation*}
By combining density and \eqref{eq:kernel-form-identity}, we can conclude that $T$ represents the form.

Conversely, if $S$ represents the bounded form, testing against $K_z^H$ gives
\begin{equation*}
 (Sf)(z)=\inner{Sf}{K_z^H}_H
 =\int_{\C^n}f(w)\overline{K_z^H(w)}e^{-\abs{w}^2/\alpha}\,d\mu(w),
\end{equation*}
with absolute convergence by the form bound.
\end{proof}

\begin{lemma}\label{lem:holomorphic-criterion}
Let $\mu$ be a positive Borel measure on $\C^n$, let $0<p<\infty$, and let $r>0$. The form in \eqref{eq:toeplitz-form-intro} is bounded on $\F^2_\alpha$ and $T_\mu\in\Sp_p$ if and only if $z\mapsto\mu(B(z,r))$ belongs to $L^p(\C^n,dV)$. This condition does not depend on $r$.
\end{lemma}

\begin{proof}
Let $c_\alpha=(\pi\alpha)^{-n}$ and let $\inner{\cdot}{\cdot}_0$ be the unnormalized inner product in \cite{IsralowitzVirtanenWolf2015}, with $2\phi(z)=\abs{z}^2/\alpha$. Reproduction and the kernel formula give
\[
 \inner{f}{g}_{2,\alpha}=c_\alpha\inner{f}{g}_0,
 \qquad K_z^0=c_\alpha K_z,
 \qquad T_\mu^0=c_\alpha T_\mu.
\]
Thus the normalization does not affect boundedness or Schatten membership. Moreover,
\[
 dd^c\phi=\frac{1}{2\alpha}dd^c\abs{z}^2,
\]
so the weight hypothesis in that paper holds by choosing $0<c<1/(2\alpha)<C$.

The desired equivalence and the independence of $r$ follow from \cite[Theorems~2.7 and~3.2]{IsralowitzVirtanenWolf2015} once we verify their auxiliary integrability hypothesis; see also \cite[Theorems~4.4 and~5.4]{IsralowitzZhu2010} in one dimension.

We use the elementary implication
\begin{equation}\label{eq:uniform-mass-implies-ivw}
 \sup_{a\in\C^n}\mu(B(a,\rho))<\infty
 \quad\to\quad
 \int_{\C^n}e^{-\gamma\abs{z-w}}\,d\mu(w)<\infty
\end{equation}
for every $\rho,\gamma>0$ and $z\in\C^n$. Indeed, partition $\C^n\simeq\mathbb R^{2n}$ into cubes of side length $\rho/\sqrt{2n}$. If $a_j$ is the center of $Q_j$, then $Q_j\subset B(a_j,\rho)$. If the supremum on the left is $M_\rho$, then a shell count gives
\[
 \int_{\C^n}e^{-\gamma\abs{z-w}}\,d\mu(w)
 \leq C_{n,\rho}M_\rho e^{\gamma\rho}
 \sum_{k=0}^\infty(k+1)^{2n}e^{-\gamma k}<\infty.
\]
Thus $\mu$ satisfies the integrability condition \cite[(2.1)]{IsralowitzVirtanenWolf2015}.

Suppose first that $w\mapsto\mu(B(w,r))$ belongs to $L^p(\C^n,dV)$. For $w\in B(z,r/2)$, we have $B(z,r/2)\subset B(w,r)$; integrating the resulting inequality over $w\in B(z,r/2)$,
\[
 \sup_{z\in\C^n}\mu(B(z,r/2))
 \leq V(B(0,r/2))^{-1/p}
 \norm{\mu(B(\cdot,r))}_{L^p(\C^n,dV)}<\infty.
\]
Thus \eqref{eq:uniform-mass-implies-ivw} applies.

Conversely, suppose that $T_\mu\in\Sp_p$. The normalized kernel is
\[
 k_z(w)=\frac{K_z(w)}{\norm{K_z}_{2,\alpha}}
 =e^{w\cdot\overline z/\alpha-\abs{z}^2/(2\alpha)}.
\]
Completing the square and using $\norm{k_z}_{2,\alpha}=1$ gives
\[
 e^{-r^2/\alpha}\mu(B(z,r))
 \leq\int_{\C^n}e^{-\abs{w-z}^2/\alpha}\,d\mu(w)
 =\inner{T_\mu k_z}{k_z}\leq\norm{T_\mu}.
\]
Thus the local masses are uniformly bounded, and \eqref{eq:uniform-mass-implies-ivw} applies again.
\end{proof}

\section{A principle for positive blocks}\label{sec:positive-block}

Positivity controls the full form through its diagonal restrictions. We write the representing operator as $B=B_1+B_2$, where $B_j=B^{1/2}P_jB^{1/2}$ is positive and has the same nonzero eigenvalues, with multiplicities, as the corresponding diagonal block. We prove these assertions below, starting with boundedness and compactness. Finite-dimensional versions appear in~\cite[Lemma~1.1 and equation~(1.1)]{BourinLee2013}; the concave trace estimate in~\cite[Corollary~2.7]{BourinLee2013} uses Rotfel'd's inequality~\cite{Rotfeld1969}. Here the $p$-triangle inequality~\eqref{eq:p-triangle} gives the range $0<p<1$ directly.

For $0<p\leq1$, the Schatten quasi-norm satisfies the $p$-triangle inequality
\begin{equation}\label{eq:p-triangle}
 \norm{X+Y}_{\Sp_p}^p
 \leq\norm{X}_{\Sp_p}^p+\norm{Y}_{\Sp_p}^p
\end{equation}
for all compact operators $X$ and $Y$; see \cite{McCarthy1967}.

\begin{proposition}[Positive block principle]
\label{prop:positive-block}
Let $m\geq1$, let
\[
 H=H_1\oplus\cdots\oplus H_m
\]
be a finite orthogonal sum of separable complex Hilbert spaces, and let $q$ be a positive sesquilinear form defined on all of $H$. Let $\iota_j:H_j\to H$ be the coordinate inclusion and $P_j=\iota_j\iota_j^*$ the corresponding orthogonal projection. Suppose that, for every $j$, the restriction of $q$ to $H_j$ is bounded and represented by a positive operator $A_j$:
\[
 q(\iota_jx,\iota_jy)=\inner{A_jx}{y}_{H_j},
 \qquad x,y\in H_j.
\]
Then $q$ is bounded on $H$. If $B$ is its representing operator, then
\begin{equation}\label{eq:block-boundedness}
 0\leq B\leq m(A_1\oplus\cdots\oplus A_m)
\end{equation}
and
\begin{equation}\label{eq:block-operator-norm}
\norm{B}\leq\sum_{j=1}^m\norm{A_j}.
\end{equation}
Moreover,
\begin{equation}\label{eq:block-compressions}
\iota_j^*B\iota_j=A_j,
\qquad 1\leq j\leq m.
\end{equation}

The positive operators
\begin{equation}\label{eq:block-factorization}
 B_j=B^{1/2}P_jB^{1/2},
 \qquad 1\leq j\leq m,
\end{equation}
satisfy
\[
 B=\sum_{j=1}^mB_j.
\]
The operator $B$ is compact if and only if every $A_j$ is compact. When these operators are compact,
\begin{equation}\label{eq:block-compression}
 s_k(A_j)\leq s_k(B),
 \qquad k\geq1,\quad 1\leq j\leq m,
\end{equation}
and, for every $N\geq1$,
\begin{equation}\label{eq:block-ky-fan}
 \sum_{k=1}^Ns_k(B)
 \leq\sum_{j=1}^m\sum_{k=1}^Ns_k(A_j).
\end{equation}

As a result, for every symmetrically normed ideal $(\mathcal I_\Phi,\norm{\cdot}_\Phi)$,
\[
 B\in\mathcal I_\Phi
 \quad\Longleftrightarrow\quad
 A_j\in\mathcal I_\Phi\quad\text{for every }j,
\]
and, whenever these conditions hold,
\begin{equation}\label{eq:block-symmetric-norm}
 \max_{1\leq j\leq m}\norm{A_j}_\Phi
 \leq\norm{B}_\Phi
 \leq\sum_{j=1}^m\norm{A_j}_\Phi.
\end{equation}
For the Schatten ideals this gives
\begin{align}
 \max_{1\leq j\leq m}\norm{A_j}_{\Sp_p}^p
 &\leq\norm{B}_{\Sp_p}^p
 \leq\sum_{j=1}^m\norm{A_j}_{\Sp_p}^p,
 &&0<p<1,\label{eq:block-schatten-small}\\
 \max_{1\leq j\leq m}\norm{A_j}_{\Sp_p}
 &\leq\norm{B}_{\Sp_p}
 \leq\sum_{j=1}^m\norm{A_j}_{\Sp_p},
 &&1\leq p<\infty.\label{eq:block-schatten}
\end{align}
Finally,
\begin{equation}\label{eq:block-trace}
 \operatorname{Tr}B=\sum_{j=1}^m\operatorname{Tr}A_j,
\end{equation}
where the traces are understood in the extended sense.
\end{proposition}

\begin{proof}
Since $q$ is positive, it is Hermitian and satisfies the Cauchy--Schwarz inequality for forms
\begin{equation}\label{eq:block-form-cauchy-schwarz}
 \abs{q(x,y)}^2\leq q(x,x)q(y,y).
\end{equation}
Indeed, positivity gives
\[
 0\leq q(x+\lambda y,x+\lambda y)
 =q(x,x)+2\operatorname{Re}\bigl(\overline\lambda q(x,y)\bigr)
   +|\lambda|^2q(y,y)
 \qquad (\lambda\in\C).
\]
If $q(y,y)>0$, choosing $\lambda=-q(x,y)/q(y,y)$ proves the Cauchy--Schwarz inequality for forms. If $q(y,y)=0$, choosing $\lambda=-tq(x,y)$ gives $0\leq q(x,x)-2t|q(x,y)|^2$ for every $t>0$, which implies that $q(x,y)=0$.

Write $x=x_1+\cdots+x_m$, where $x_j\in H_j$, and put $a_j=q(x_j,x_j)$. By \eqref{eq:block-form-cauchy-schwarz},
\begin{align*}
 q(x,x)
 &\leq\sum_{j,\ell=1}^m\abs{q(x_j,x_\ell)}\\
 &\leq\left(\sum_{j=1}^m\sqrt{a_j}\right)^2\\
 &\leq m\sum_{j=1}^ma_j\\
 &=m\sum_{j=1}^m\inner{A_jx_j}{x_j}_{H_j}.
\end{align*}
Also,
\begin{align*}
 q(x,x)
 &\leq\left(\sum_{j=1}^m
      \sqrt{\norm{A_j}}\norm{x_j}\right)^2\\
 &\leq\left(\sum_{j=1}^m\norm{A_j}\right)\norm{x}^2.
\end{align*}
Together with~\eqref{eq:block-form-cauchy-schwarz}, this gives
\[
 \abs{q(x,y)}
 \leq\left(\sum_{j=1}^m\norm{A_j}\right)\norm{x}\norm{y}.
\]
Thus $q$ is bounded. Let $B$ be its representing operator; the two quadratic form estimates give \eqref{eq:block-boundedness} and \eqref{eq:block-operator-norm}. For $x,y\in H_j$, note that
\[
 \inner{\iota_j^*B\iota_jx}{y}_{H_j}
 =q(\iota_jx,\iota_jy)=\inner{A_jx}{y}_{H_j},
\]
which proves \eqref{eq:block-compressions}.

Since $\sum_jP_j=I$, equation \eqref{eq:block-factorization} gives $B=\sum_jB_j$. Letting $Y_j=P_jB^{1/2}$,
\begin{equation}\label{eq:block-products}
 Y_j^*Y_j=B_j,
 \qquad
 Y_jY_j^*=P_jBP_j=\iota_jA_j\iota_j^*.
\end{equation}
For each $j$, one product in \eqref{eq:block-products} is compact if and only if the other is. In that case they have the same nonzero eigenvalues, including multiplicity. Hence, in the compact case,
\begin{equation}\label{eq:block-singular-identities}
 s_k(B_j)=s_k(A_j),
 \qquad k\geq1.
\end{equation}
If every $A_j$ is compact, then every $B_j$ and their finite sum $B$ are compact. The converse follows from $A_j=\iota_j^*B\iota_j$.

When the operators are compact, the ideal inequality
\[
 s_k(XTY)\leq\norm{X}\norm{Y}s_k(T)
\]
for compact $T$ and bounded $X,Y$ \cite[Theorem~1.6]{Simon2005} gives $s_k(A_j)\leq s_k(B)$, proving \eqref{eq:block-compression}. Ky Fan's inequality \cite{Fan1951}, $B=\sum_jB_j$, and \eqref{eq:block-singular-identities} give \eqref{eq:block-ky-fan}.

If every $A_j$ belongs to $\mathcal I_\Phi$, then \eqref{eq:block-singular-identities} and the triangle inequality give
\[
 \norm{B}_\Phi
 \leq\sum_{j=1}^m\norm{B_j}_\Phi
 =\sum_{j=1}^m\norm{A_j}_\Phi.
\]
Conversely, if $B\in\mathcal I_\Phi$, then $A_j=\iota_j^*B\iota_j\in\mathcal I_\Phi$ and $\norm{A_j}_\Phi\leq\norm{B}_\Phi$. This proves the ideal equivalence and \eqref{eq:block-symmetric-norm}.

For $0<p<1$, the $p$-triangle inequality \eqref{eq:p-triangle} gives
\begin{align*}
 \norm{B}_{\Sp_p}^p
 &=\norm{\textstyle\sum_{j=1}^mB_j}_{\Sp_p}^p\\
 &\leq\sum_{j=1}^m\norm{B_j}_{\Sp_p}^p\\
 &=\sum_{j=1}^m\norm{A_j}_{\Sp_p}^p.
\end{align*}
For $1\leq p<\infty$, equation~\eqref{eq:block-schatten} is the Schatten case of~\eqref{eq:block-symmetric-norm}. The lower estimates follow from~\eqref{eq:block-compression}.

Finally, for orthonormal bases $\mathcal E_j$ of $H_j$, Tonelli's theorem gives
\begin{align*}
 \operatorname{Tr}B
 &=\sum_{j=1}^m\sum_{e\in\mathcal E_j}
   \inner{B\iota_je}{\iota_je}_H\\
 &=\sum_{j=1}^m\sum_{e\in\mathcal E_j}
   \inner{A_je}{e}_{H_j}\\
 &=\sum_{j=1}^m\operatorname{Tr}A_j.
\end{align*}
This proves \eqref{eq:block-trace}.
\end{proof}

\section{Application to pluriharmonic Fock space}
\label{sec:application}

To apply Proposition~\ref{prop:positive-block}, it remains to identify the second diagonal block relative to \eqref{eq:orthogonal-splitting}.

\begin{lemma}[The second diagonal block]\label{lem:second-diagonal-block}
Suppose that the holomorphic Toeplitz form is bounded, and write $A=T_\mu$. Let $P_0$ be the orthogonal projection of $\F^2_\alpha$ onto $\F^2_{\alpha,0}$, and set
\[
 A_0=P_0A|_{\F^2_{\alpha,0}}.
\]
Then the integral in \eqref{eq:toeplitz-form-intro} is absolutely convergent for every pair of functions in $\PH^2_\alpha$. Relative to \eqref{eq:orthogonal-splitting}, its two diagonal restrictions are represented by
\[
 A\quad\text{and}\quad D=JA_0J^{-1},
\]
where
\[
 J:\F^2_{\alpha,0}\to\overline{\F^2_{\alpha,0}},
 \qquad Jh=\overline h,
\]
is antiunitary. If $A$ is compact, then $D$ is compact and
\begin{equation}\label{eq:second-block-singular-values}
 s_k(D)=s_k(A_0)\leq s_k(A),
 \qquad k\geq1.
\end{equation}
\end{lemma}

\begin{proof}
For $f=g+\overline h\in\PH^2_\alpha$, the pointwise inequality $\abs{g+\overline h}^2\leq2(\abs g^2+\abs h^2)$ and~\eqref{eq:norm-splitting} give
\begin{equation*}
 \int_{\C^n}\abs{f(z)}^2e^{-\abs{z}^2/\alpha}\,d\mu(z)
 \leq2\inner{Ag}{g}+2\inner{Ah}{h}
 \leq2\norm A\norm{f}_{2,\alpha}^2.
\end{equation*}
Weighted Cauchy--Schwarz proves absolute convergence and boundedness of the full form. Its first diagonal restriction is represented by $A$.

For $h,k\in\F^2_{\alpha,0}$, positivity of $\mu$ and antiunitarity of $J$ give
\begin{align*}
 \int_{\C^n}\overline{h(z)}k(z)e^{-\abs{z}^2/\alpha}\,d\mu(z)
 &=\overline{\inner{A_0h}{k}}\\
 &=\inner{JA_0h}{Jk}.
\end{align*}
Thus the second diagonal restriction is represented by the linear operator $D=JA_0J^{-1}$. If $A$ is compact, compression gives $s_k(A_0)\leq s_k(A)$, and antiunitary conjugation preserves the eigenvalues of a positive compact operator, including multiplicity. This proves~\eqref{eq:second-block-singular-values}.
\end{proof}

\begin{proof}[Proof of Theorem~\ref{thm:main}]
By Lemma~\ref{lem:holomorphic-criterion}, it remains to compare the two Toeplitz forms. Suppose first that $A=T_\mu\in\Sp_p$. Lemma~\ref{lem:second-diagonal-block} shows that the pluriharmonic form is bounded, with diagonal operators $A$ and $D=JA_0J^{-1}$ satisfying
\begin{equation}\label{eq:second-block-schatten}
 \norm{D}_{\Sp_p}^p
 =\norm{A_0}_{\Sp_p}^p
 \leq\norm{A}_{\Sp_p}^p.
\end{equation}
Proposition~\ref{prop:positive-block}, applied with $m=2$, gives
\begin{align*}
 \norm{T_\mu^{\mathrm{ph}}}_{\Sp_p}^p
 &\leq\norm A_{\Sp_p}^p+\norm D_{\Sp_p}^p
 \leq2\norm A_{\Sp_p}^p,
 &&0<p<1,\\
 \norm{T_\mu^{\mathrm{ph}}}_{\Sp_p}
 &\leq\norm A_{\Sp_p}+\norm D_{\Sp_p}
 \leq2\norm A_{\Sp_p},
 &&p\geq1.
\end{align*}
These are the upper estimates in~\eqref{eq:quantitative-estimate}.

Conversely, suppose that the first condition holds. Restriction to the holomorphic subspace $F^2_\alpha$ gives a bounded form represented by
\begin{equation*}
 T_\mu=\iota^*T_\mu^{\mathrm{ph}}\iota,
 \qquad \iota:\F^2_\alpha\hookrightarrow\PH^2_\alpha.
\end{equation*}
Since $T_\mu$ is the compression of $T_\mu^{\mathrm{ph}}$ to the holomorphic subspace $F^2_\alpha$, the singular values of these two operators obey, for each $k$, \[ s_k(T_\mu)\leq s_k(T_\mu^{\mathrm{ph}})\]
Summing their $p$th powers shows that $T_\mu^{\mathrm{ph}}\in\Sp_p$ implies $T_\mu\in\Sp_p$, with $\norm{T_\mu}_{\Sp_p}^p\leq\norm{T_\mu^{\mathrm{ph}}}_{\Sp_p}^p$.
\end{proof}

\begin{proof}[Proof of Corollary~\ref{cor:symmetric-ideals}]
Use the same diagonal operators $A=T_\mu$ and $D=JA_0J^{-1}$. If $A\in\mathcal I_\Phi$, then~\eqref{eq:second-block-singular-values} gives $\norm D_\Phi\leq\norm A_\Phi$. Proposition~\ref{prop:positive-block} yields
\begin{equation*}
 \norm{T_\mu^{\mathrm{ph}}}_\Phi
 \leq\norm A_\Phi+\norm D_\Phi
 \leq2\norm A_\Phi,
\end{equation*}
and~\eqref{eq:block-ky-fan} gives~\eqref{eq:ky-fan-comparison}. The converse and the lower estimate follow from the holomorphic compression, as in the proof of Theorem~\ref{thm:main}.

For every normalized symmetric norming function, a positive operator of rank one has norm equal to its nonzero eigenvalue. Proposition~\ref{prop:sharp-constants} therefore gives
\[
 \frac{\norm{T_{\delta_a}^{\mathrm{ph}}}_\Phi}
 {\norm{T_{\delta_a}}_\Phi}
 =2-e^{-\abs a^2/\alpha}.
\]
Equality in the lower bound holds at $a=0$, and the ratio tends to $2$ as $\abs a\to\infty$.
\end{proof}

\section{Trace and sharpness}
\label{sec:sharpness}

The pluriharmonic trace identity is contained in the proof of~\cite[Proposition~3.3]{JaguzovicVujadinovic2026}. We include a proof and determine the possible ratios of the two traces.

\begin{proposition}[Trace identity]\label{prop:trace-identity}
Let $\mu$ be a positive Borel measure on $\C^n$. The following are equivalent:
\begin{enumerate}
\item $\mu(\C^n)<\infty$;
\item the holomorphic form is bounded and $T_\mu\in\Sp_1$;
\item the pluriharmonic form is bounded and $T_\mu^{\mathrm{ph}}\in\Sp_1$.
\end{enumerate}
Whenever these conditions hold,
\begin{align}
 \norm{T_\mu}_{\Sp_1}
 &=\mu(\C^n),\label{eq:holomorphic-trace}\\
 \norm{T_\mu^{\mathrm{ph}}}_{\Sp_1}
 &=2\mu(\C^n)
 -\int_{\C^n}e^{-\abs{z}^2/\alpha}\,d\mu(z).
 \label{eq:pluriharmonic-trace}
\end{align}

If $\mu\neq0$, then
\begin{equation}\label{eq:trace-ratio}
 \frac{\norm{T_\mu^{\mathrm{ph}}}_{\Sp_1}}
 {\norm{T_\mu}_{\Sp_1}}
 =2-\frac{1}{\mu(\C^n)}
 \int_{\C^n}e^{-\abs{z}^2/\alpha}\,d\mu(z).
\end{equation}
For nonzero finite positive measures $\mu$, the set of all possible values of this ratio is exactly the half-open interval $[1,2)$ with the value $1$ occuring if and only if $\mu=c\delta_0$ for some $c>0$.
\end{proposition}

\begin{proof}
Suppose first that $\mu(\C^n)<\infty$. The reproducing-kernel estimate and~\eqref{eq:ph-kernel} give, on either space $H$,
\begin{equation*}
 \abs{f(z)}^2e^{-\abs z^2/\alpha}
 \leq\norm f_H^2K_z^H(z)e^{-\abs z^2/\alpha}
 \leq2\norm f_H^2.
\end{equation*}
Integrating the preceding estimate and applying Cauchy--Schwarz with respect to the weighted measure, we see that both forms are bounded. The monomial basis $(e_m)$ from Section~\ref{sec:splitting} satisfies
\[
 \sum_m\abs{e_m(z)}^2=e^{\abs{z}^2/\alpha}.
\]
By Tonelli,
\begin{align*}
 \operatorname{Tr}T_\mu
 &=\int_{\C^n}e^{-\abs z^2/\alpha}
   \sum_{m\in\mathbb N^n}\abs{e_m(z)}^2\,d\mu(z)\\
 &=\mu(\C^n).
\end{align*}
This is the holomorphic trace formula; see~\cite[Proposition~4.1]{IsralowitzZhu2010} in one dimension. Since $e_0=1$, the compression $A_0$ to $\F^2_{\alpha,0}$ satisfies
\begin{equation*}
 \operatorname{Tr}A_0
 =\mu(\C^n)-\int_{\C^n}e^{-\abs z^2/\alpha}\,d\mu(z).
\end{equation*}
Since the second diagonal block has the same trace as $A_0$, i.e. $\operatorname{Tr}A_0$, the trace identity~\eqref{eq:block-trace} gives~\eqref{eq:pluriharmonic-trace}. The positivity of $T_\mu$ and $T_\mu^{\mathrm{ph}}$ implies that their traces equal their trace norms.

Conversely, if $T_\mu\in\Sp_1$, summing over the monomial basis and applying Tonelli's theorem gives
\[
 \mu(\C^n)=\operatorname{Tr}T_\mu<\infty.
\]
If $T_\mu^{\mathrm{ph}}\in\Sp_1$, then its compression $T_\mu$ to the holomorphic subspace is also trace class, so $\mu$ is finite in this case as well.

For $\mu\ne0$, equation~\eqref{eq:trace-ratio} follows by division. Since $0<e^{-\abs z^2/\alpha}\leq1$, the ratio lies in $[1,2)$ and equals $1$ exactly when $\mu$ is concentrated at the origin. Every $c\in[1,2)$ is attained by $\delta_a$ with $e^{-\abs a^2/\alpha}=2-c$.
\end{proof}

\begin{proposition}\label{prop:sharp-constants}
For every $n\geq1$ and $0<p<\infty$, both constants in \eqref{eq:quantitative-estimate} are optimal.
\end{proposition}

\begin{proof}
Let $\delta_a$ denote the unit point mass at $a\in\C^n$. Taking $\mu=\delta_0$, we have $K_0=K_0^{\mathrm{ph}}=1$, so both Toeplitz operators are the projection of rank one onto the constant functions. Therefore
\[
 \norm{T_{\delta_0}}_{\Sp_p}^p
 =\norm{T_{\delta_0}^{\mathrm{ph}}}_{\Sp_p}^p=1,
\]
so the constant in the left inequality cannot be larger than $1$.

For a point mass at $a$, the operator on either reproducing kernel space $H$ is
\begin{equation*}
 f\mapsto e^{-\abs a^2/\alpha}\inner{f}{K_a^H}_H K_a^H.
\end{equation*}
Its only nonzero eigenvalue is $e^{-\abs a^2/\alpha}K_a^H(a)$, which equals $1$ in the holomorphic space and $2-e^{-\abs a^2/\alpha}$ in the pluriharmonic space. Hence
\begin{equation*}
 \frac{\norm{T_{\delta_a}^{\mathrm{ph}}}_{\Sp_p}^p}
 {\norm{T_{\delta_a}}_{\Sp_p}^p}
 =\left(2-e^{-\abs a^2/\alpha}\right)^p\to2^p
 \qquad (\abs a\to\infty).
\end{equation*}
This proves sharpness of the upper constant for $p\geq1$.

For $0<p<1$, we must make the ratio approach $2$, whereas point masses give only $2^p$. Toward this end we use normalized arclength measures on circles centred at the origin in the first coordinate line. For these measures, $z_1^k$ and $\overline{z_1}^k$ are eigenfunctions with the same positive eigenvalue for every $k\geq1$, so an eigenvalue $\lambda$ associated with $z_1^k$ contributes $\lambda^p$ to the holomorphic Schatten sum and $2\lambda^p$ to the pluriharmonic sum. 

For a circle of radius $R$, the constant function contributes $e^{-pR^2/\alpha}$ to each Schatten sum, while the term corresponding to $z_1$ contributes $e^{-pR^2/\alpha}(R^2/\alpha)^p$. We see that constant term divided by the full holomorphic sum is at most $(\alpha/R^2)^p$, which approaches 0 as $R\to\infty$. Since the pluriharmonic sum is twice the holomorphic sum minus this single constant term, their ratio tends to $2$.

For $R>0$, let $\sigma_R$ be normalized arclength measure on
\[
 \{(Re^{i\theta},0,\ldots,0):0\leq\theta<2\pi\},
\]
and put $t=R^2/\alpha$. For $k\geq0$, set
\[
 e_k(z)=\frac{z_1^k}{\alpha^{k/2}\sqrt{k!}};
\]
then for $j,k\geq0$, direct integration gives
\begin{align*}
 &\int_{\C^n}e_j(z)\overline{e_k(z)}
 e^{-\abs{z}^2/\alpha}\,d\sigma_R(z)\\
 &\quad=\frac{e^{-t}t^{(j+k)/2}}{\sqrt{j!k!}}
 \frac{1}{2\pi}\int_0^{2\pi}e^{i(j-k)\theta}\,d\theta\\
 &\quad=\delta_{jk}e^{-t}\frac{t^k}{k!}.
\end{align*}
Every other monomial basis vector vanishes on the support of $\sigma_R$.
Hence the nonzero eigenvalues of $T_{\sigma_R}$ are
\[
 a_k=e^{-t}\frac{t^k}{k!},
 \qquad k\geq0.
\]
The sequence $(a_k)_{k\geq0}$ belongs to $\ell^p$ for every $p>0$, since
\[
 \frac{a_{k+1}^p}{a_k^p}
 =\left(\frac{t}{k+1}\right)^p\to0.
\]
The same calculation on the antiholomorphic summand gives the eigenvalues $a_k$, $k\geq1$. Moreover, for $j\geq0$ and $k\geq1$,
\begin{align*}
 &\int_{\C^n}e_j(z)e_k(z)e^{-\abs{z}^2/\alpha}
 \,d\sigma_R(z)\\
 &\quad=\frac{e^{-t}t^{(j+k)/2}}{\sqrt{j!k!}}
 \frac{1}{2\pi}\int_0^{2\pi}e^{i(j+k)\theta}\,d\theta=0.
\end{align*}
Thus both mixed blocks vanish, and the nonzero eigenvalues of $T_{\sigma_R}^{\mathrm{ph}}$ are $a_0$ once and $a_k$ twice for $k\geq1$. Therefore
\begin{align*}
 \frac{\norm{T_{\sigma_R}^{\mathrm{ph}}}_{\Sp_p}^p}
 {\norm{T_{\sigma_R}}_{\Sp_p}^p}
 &=\frac{a_0^p+2\sum_{k=1}^\infty a_k^p}
 {\sum_{k=0}^\infty a_k^p}\\
 &=2-\frac{a_0^p}{\sum_{k=0}^\infty a_k^p}\\
 &=2-\frac{1}{\displaystyle
 \sum_{k=0}^\infty\frac{t^{pk}}{(k!)^p}}.
\end{align*}
Since
\[
 \sum_{k=0}^\infty\frac{t^{pk}}{(k!)^p}\geq1+t^p
 \to\infty,
\]
this ratio tends to $2$ as $R\to\infty$. Hence $2$ is optimal for $0<p<1$.
\end{proof}

\section{Generalized weights and the inherited norm}
\label{sec:generalized-weights}

\subsection{The canonical direct sum}

Let
\[
 d^c=\frac{i}{4}(\overline\partial-\partial),
 \qquad \omega_0=dd^c\abs{z}^2
\]
and let $\phi\in C^2(\C^n;\mathbb R)$ satisfy
\begin{equation}\label{eq:generalized-weight}
 c\omega_0\leq dd^c\phi\leq C\omega_0
\end{equation}
for some $c,C>0$. The generalized holomorphic Fock space is
\[
 \F_\phi^2
 =\left\{g\in\operatorname{Hol}(\C^n):
 \norm{g}_\phi^2
 =\int_{\C^n}\abs{g(z)}^2e^{-2\phi(z)}\dV(z)<\infty\right\}.
\]
No radiality or circle invariance is assumed.

Let
\[
 \mathcal C_\phi=\F_\phi^2\cap\mathbb C,
 \qquad
 \F_{\phi,\circ}^2=\F_\phi^2\cap(\mathcal C_\phi)^\perp,
\]
so that $\mathcal C_\phi$ is either $\{0\}$ or $\mathbb C1$. We define
\begin{equation}\label{eq:generalized-direct-sum}
 \PH_{\phi,\oplus}^2
 =\left\{g+\overline h:
 g\in\F_\phi^2,\ h\in\F_{\phi,\circ}^2\right\}
\end{equation}
and equip this space with the norm
\begin{equation}\label{eq:generalized-direct-sum-norm}
 \norm{g+\overline h}_{\phi,\oplus}^2
 =\norm{g}_\phi^2+\norm{h}_\phi^2.
\end{equation}
The representation in \eqref{eq:generalized-direct-sum} is unique: if $g+\overline h=0$, then $g=-\overline h$ is constant; the constant vanishes either because constants do not belong to $\F_\phi^2$, or because $h$ is orthogonal to $\mathcal C_\phi$. Hence one has the orthogonal direct sum
\[
 \PH_{\phi,\oplus}^2 =\F_\phi^2\oplus\overline{\F_{\phi,\circ}^2}.
\]
The direct sum inner product need not agree with the inner product inherited from $L^2(\C^n,e^{-2\phi}dV)$.

For a positive Borel measure $\mu$, consider the form
\begin{equation}\label{eq:generalized-toeplitz-form}
 q_{\mu,\phi}(f,u)
 =\int_{\C^n}f(z)\overline{u(z)}e^{-2\phi(z)}\,d\mu(z).
\end{equation}
If the Toeplitz form is bounded on $\F_\phi^2$, we denote its representing operator by $T_{\mu,\phi}$. If it is bounded on $\PH_{\phi,\oplus}^2$, we denote the representing operator on that space by $T_{\mu,\phi}^{\mathrm{ph},\oplus}$.

\begin{theorem}[Direct sum criterion]
\label{thm:generalized-direct-sum}
Let $\phi$ satisfy \eqref{eq:generalized-weight}, let $\mu$ be a positive Borel measure on $\C^n$, let $0<p<\infty$, and fix $r>0$. The following statements are equivalent:
\begin{enumerate}
\item The form \eqref{eq:generalized-toeplitz-form} is bounded on $\PH_{\phi,\oplus}^2$, and $T_{\mu,\phi}^{\mathrm{ph},\oplus}\in\Sp_p$.
\item The holomorphic form is bounded on $\F_\phi^2$, and $T_{\mu,\phi}\in\Sp_p$.
\item The function $z\mapsto\mu(B(z,r))$ belongs to $L^p(\C^n,dV)$.
\end{enumerate}
The third condition is independent of $r>0$. Whenever these conditions hold,
\begin{equation}\label{eq:generalized-schatten-comparison}
 \norm{T_{\mu,\phi}}_{\Sp_p}^p
 \leq\norm{T_{\mu,\phi}^{\mathrm{ph},\oplus}}_{\Sp_p}^p
 \leq2^{\max\{1,p\}}\norm{T_{\mu,\phi}}_{\Sp_p}^p.
\end{equation}

For every symmetrically normed ideal $(\mathcal I_\Phi,\norm{\cdot}_\Phi)$, membership of $T_{\mu,\phi}$ and $T_{\mu,\phi}^{\mathrm{ph},\oplus}$ is also equivalent, with boundedness of the forms understood, and
\begin{equation}\label{eq:generalized-symmetric-comparison}
 \norm{T_{\mu,\phi}}_\Phi
 \leq\norm{T_{\mu,\phi}^{\mathrm{ph},\oplus}}_\Phi
 \leq2\norm{T_{\mu,\phi}}_\Phi.
\end{equation}
More precisely, for every $N\geq1$,
\begin{equation}\label{eq:generalized-ky-fan}
 \sum_{k=1}^Ns_k(T_{\mu,\phi}^{\mathrm{ph},\oplus})
 \leq2\sum_{k=1}^Ns_k(T_{\mu,\phi}).
\end{equation}
All the constants above are optimal uniformly over the class of weights satisfying \eqref{eq:generalized-weight}.
\end{theorem}

\begin{proof}
The holomorphic local mass criterion is due to \cite[Theorems~2.7 and~3.2]{IsralowitzVirtanenWolf2015}. Their theorem assumes that
\begin{equation}\label{eq:generalized-exponential-integrability}
 \int_{\C^n}e^{-\gamma\abs{z-w}}\,d\mu(w)<\infty
 \qquad (z\in\C^n,\ \gamma>0).
\end{equation}
This condition is automatic on either side of the equivalence. If the holomorphic form is bounded, testing it on normalized reproducing kernels and using the near-diagonal lower estimate in \cite[Lemma~2.2]{IsralowitzVirtanenWolf2015} gives
\[
\sup_{z\in\C^n}\mu(B(z,\delta))<\infty
\]
for some $\delta>0$. If the local mass function belongs to $L^p$, the ball comparison in the proof of Lemma~\ref{lem:holomorphic-criterion} gives the same uniform bound. In either case, \eqref{eq:uniform-mass-implies-ivw} verifies \eqref{eq:generalized-exponential-integrability}. Under this integrability condition, the proof of Lemma~\ref{lem:kernel-form-agreement} applies verbatim to $\F_\phi^2$ and hence the Toeplitz operator in \cite{IsralowitzVirtanenWolf2015} (defined using kernels) agrees with the operator representing the holomorphic form.

For the comparison, suppose that the holomorphic form is bounded and put $A=T_{\mu,\phi}$. Let $P_\circ$ be the orthogonal projection from $\F_\phi^2$ onto $\F_{\phi,\circ}^2$, and set
\[
 A_\circ=P_\circ A|_{\F_{\phi,\circ}^2}.
\]
Conjugation defines an antiunitary map
\[
 J:\F_{\phi,\circ}^2\to\overline{\F_{\phi,\circ}^2},
 \qquad Jh=\overline h,
\]
where the target has the direct sum norm.

For $f=g+\overline h\in\PH_{\phi,\oplus}^2$,
\begin{align*}
 q_{\mu,\phi}(f,f)
 &\leq2\int_{\C^n}
 \left(\abs{g(z)}^2+\abs{h(z)}^2\right)e^{-2\phi(z)}\,d\mu(z)\\
 &\leq2\norm{A}
 \left(\norm{g}_\phi^2+\norm{h}_\phi^2\right).
\end{align*}
Thus the full form is bounded and its two diagonal restrictions are represented by
\[
 A \quad\text{and}\quad
 D=JA_\circ J^{-1}.
\]
If $A$ is compact, then
\begin{equation}\label{eq:generalized-second-block}
 s_k(D)=s_k(A_\circ)\leq s_k(A),
 \qquad k\geq1.
\end{equation}
Proposition~\ref{prop:positive-block}, applied with $m=2$, now gives the upper estimate in \eqref{eq:generalized-schatten-comparison}, \eqref{eq:generalized-symmetric-comparison}, and \eqref{eq:generalized-ky-fan}.

Conversely, let
\[
 \iota:\F_\phi^2\to\PH_{\phi,\oplus}^2
\]
be inclusion into the first summand. Whenever the full form is bounded,
\begin{equation}\label{eq:generalized-holomorphic-compression}
 T_{\mu,\phi}
 =\iota^*T_{\mu,\phi}^{\mathrm{ph},\oplus}\iota.
\end{equation}
The ideal inequality for singular values proves the converse implications and the lower estimates.

The Gaussian weights belong to the class \eqref{eq:generalized-weight}. Proposition~\ref{prop:sharp-constants} and the point mass examples used in the proof of Corollary~\ref{cor:symmetric-ideals} therefore prove uniform optimality.
\end{proof}

Let $K_\phi(z,w)$ be the reproducing kernel of $\F_\phi^2$, and let
\begin{equation}\label{eq:generalized-kernel-density}
 \rho_\phi(z)=K_\phi(z,z)e^{-2\phi(z)}.
\end{equation}
The generalized kernel estimates give
\begin{equation}\label{eq:generalized-kernel-density-bounds}
 \rho_\phi(z)\sim1    \qquad (z\in\C^n);
\end{equation}
see \cite[Lemma~2.2]{IsralowitzVirtanenWolf2015}.

\begin{proposition}[Generalized trace identity]
\label{prop:generalized-trace}
Let $\phi$ satisfy \eqref{eq:generalized-weight}, and let $\mu$ be a positive Borel measure. The following statements are equivalent:
\begin{enumerate}
\item $\mu(\C^n)<\infty$;
\item $T_{\mu,\phi}\in\Sp_1$, with boundedness of the holomorphic form;
\item $T_{\mu,\phi}^{\mathrm{ph},\oplus}\in\Sp_1$, with boundedness of
the direct sum form.
\end{enumerate}
Whenever these conditions hold,
\begin{equation}\label{eq:generalized-holomorphic-trace}
 \operatorname{Tr}T_{\mu,\phi}
 =\int_{\C^n}\rho_\phi(z)\,d\mu(z).
\end{equation}
If $1\notin\F_\phi^2$, then
\begin{equation}\label{eq:generalized-ph-trace-no-constant}
 \operatorname{Tr}T_{\mu,\phi}^{\mathrm{ph},\oplus}
 =2\int_{\C^n}\rho_\phi(z)\,d\mu(z).
\end{equation}
If $1\in\F_\phi^2$, then
\begin{equation}\label{eq:generalized-ph-trace-with-constant}
 \operatorname{Tr}T_{\mu,\phi}^{\mathrm{ph},\oplus}
 =2\int_{\C^n}\rho_\phi(z)\,d\mu(z)
 -\frac1{M_\phi}\int_{\C^n}e^{-2\phi(z)}\,d\mu(z),
\end{equation}
where $M_\phi=\norm{1}_\phi^2=\int_{\C^n}e^{-2\phi(z)}\dV(z)$.
\end{proposition}

\begin{proof}
By Tonelli's theorem,
\[
 \int_{\C^n}\mu(B(z,r))\dV(z)
 =V(B(0,r))\mu(\C^n).
\]
The equivalence follows from Theorem~\ref{thm:generalized-direct-sum} with $p=1$.

Let $(e_j)$ be an orthonormal basis of $\F_\phi^2$. Since $K_\phi(z,z)=\sum_j\abs{e_j(z)}^2$, another application of Tonelli's theorem gives
\begin{align*}
 \operatorname{Tr}T_{\mu,\phi}
 &=\sum_j\inner{T_{\mu,\phi}e_j}{e_j}_\phi\\
 &=\int_{\C^n}K_\phi(z,z)e^{-2\phi(z)}\,d\mu(z),
\end{align*}
which is \eqref{eq:generalized-holomorphic-trace}.

Write $A=T_{\mu,\phi}$ and let $A_\circ$ be its compression to $\F_{\phi,\circ}^2$. If $1\notin\F_\phi^2$, then $A_\circ=A$. If $1\in\F_\phi^2$, the normalized constant $e_\phi=M_\phi^{-1/2}1$ gives
\begin{align*}
 \operatorname{Tr}A_\circ
 &=\operatorname{Tr}A-\inner{Ae_\phi}{e_\phi}_\phi\\
 &=\operatorname{Tr}A
 -\frac1{M_\phi}\int_{\C^n}e^{-2\phi(z)}\,d\mu(z).
\end{align*}
The second diagonal block has the same trace as $A_\circ$, so~\eqref{eq:block-trace} gives both pluriharmonic trace formulas.
\end{proof}

The orthogonal complement in the definition of $\F_{\phi,\circ}^2$ is needed here. For a nonradial weight it need not equal $\{h\in\F_\phi^2:h(0)=0\}$.

\subsection{The inherited norm}
\label{sec:inherited-norm}

There is another natural space associated with $\phi$:
\begin{equation}\label{eq:inherited-ph-space}
 \PH_{\phi,L^2}^2
 =\left\{f\in L^2(\C^n,e^{-2\phi}dV):
 f\text{ is pluriharmonic}\right\},
\end{equation}
with the inherited weighted $L^2$ norm. This is a closed subspace, since convergence in the weighted $L^2$ norm implies convergence in $L^2_{\mathrm{loc}}$. The pluriharmonic equations pass to the limit in the sense of distributions, and elliptic regularity gives a pluriharmonic representative.

The synthesis map
\begin{equation}\label{eq:generalized-synthesis-map}
 S_\phi:\PH_{\phi,\oplus}^2\to\PH_{\phi,L^2}^2,
 \qquad S_\phi(g+\overline h)=g+\overline h,
\end{equation}
is injective and bounded, since $\norm{g+\overline h}_\phi^2\leq2(\norm g_\phi^2+\norm h_\phi^2)$. It need not have closed range. 
If $S_\phi$ is an isomorphism, the two spaces consist of the same functions, and the direct sum norm is equivalent to the inherited weighted $L^2$ norm. The following proposition shows that changing between these norms preserves boundedness of the Toeplitz form and Schatten membership of its representing operator. In particular, the local mass criterion also holds for the inherited norm.

\begin{proposition}[Transfer to the inherited norm]
\label{prop:synthesis-transfer}
Suppose that $S_\phi$ in \eqref{eq:generalized-synthesis-map} is a Hilbert space isomorphism. Let $T_{\mu,\phi}^{\mathrm{ph},L^2}$ denote the operator associated with \eqref{eq:generalized-toeplitz-form} on $\PH_{\phi,L^2}^2$. Then, for every $0<p<\infty$,
\[
 T_{\mu,\phi}^{\mathrm{ph},L^2}\in\Sp_p
 \quad\Longleftrightarrow\quad
 T_{\mu,\phi}^{\mathrm{ph},\oplus}\in\Sp_p,
\]
with boundedness of the two forms understood. In particular, the three conditions in Theorem~\ref{thm:generalized-direct-sum} remain equivalent when $\PH_{\phi,\oplus}^2$ is replaced by $\PH_{\phi,L^2}^2$.

Whenever these conditions hold,
\begin{equation}\label{eq:inherited-schatten-comparison}
 \norm{T_{\mu,\phi}}_{\Sp_p}^p
 \leq\norm{T_{\mu,\phi}^{\mathrm{ph},L^2}}_{\Sp_p}^p
 \leq2^{\max\{1,p\}}\norm{S_\phi^{-1}}^{2p}
 \norm{T_{\mu,\phi}}_{\Sp_p}^p.
\end{equation}
The analogous symmetric ideal estimate is
\begin{equation}\label{eq:inherited-symmetric-comparison}
 \norm{T_{\mu,\phi}}_\Phi
 \leq\norm{T_{\mu,\phi}^{\mathrm{ph},L^2}}_\Phi
 \leq2\norm{S_\phi^{-1}}^2\norm{T_{\mu,\phi}}_\Phi.
\end{equation}
\end{proposition}

\begin{proof}
We write
\[
 B_\oplus=T_{\mu,\phi}^{\mathrm{ph},\oplus},
 \qquad
 B_{L^2}=T_{\mu,\phi}^{\mathrm{ph},L^2};
\]
these two operators represent the same form in different Hilbert norms, hence
\begin{equation}\label{eq:inherited-congruence}
 B_\oplus=S_\phi^*B_{L^2}S_\phi,
 \qquad
 B_{L^2}=(S_\phi^{-1})^*B_\oplus S_\phi^{-1}.
\end{equation}
Boundedness and membership in any operator ideal are preserved under this bounded invertible congruence. The second identity and Theorem~\ref{thm:generalized-direct-sum} give the upper estimates in \eqref{eq:inherited-schatten-comparison} and \eqref{eq:inherited-symmetric-comparison}. The lower estimates follow by compression to the isometric holomorphic subspace of~\eqref{eq:inherited-ph-space}.
\end{proof}

Scalar circle invariance makes the two summands in~\eqref{eq:generalized-direct-sum} orthogonal in the inherited norm.

\begin{proposition}[Scalar circle invariance]
\label{prop:circle-invariant-weight}
Suppose that $\phi$ satisfies \eqref{eq:generalized-weight} and
\begin{equation}\label{eq:circle-invariant-weight}
 \phi(e^{i\theta}z)=\phi(z),
 \qquad z\in\C^n,\quad \theta\in\mathbb R.
\end{equation}
Then $S_\phi$ is unitary. As a result, $\PH_{\phi,\oplus}^2=\PH_{\phi,L^2}^2$ with equality of norms, and all the conclusions and constants in Theorem~\ref{thm:generalized-direct-sum} hold for the inherited space.
\end{proposition}

\begin{proof}
For every entire function $v$ and every $z\in\C^n$, the mean value formula applied to the entire function $\lambda\mapsto v(\lambda z)$ gives
\[
 \frac1{2\pi}\int_0^{2\pi}v(e^{i\theta}z)\,d\theta=v(0).
\]
Both $B(0,R)$ and $e^{-2\phi}dV$ are invariant under $z\mapsto e^{i\theta}z$. Averaging and changing variables therefore gives
\begin{equation}\label{eq:circle-ball-mean}
 \int_{B(0,R)}v(z)e^{-2\phi(z)}\dV(z)
 =v(0)\int_{B(0,R)}e^{-2\phi(z)}\dV(z),
 \qquad R>0.
\end{equation}
All integrals here are over bounded sets, so this identity does not require $v\in\F_\phi^2$.

Let $f\in\PH_{\phi,L^2}^2$; the decomposition argument at the beginning of the proof of Lemma~\ref{lem:splitting} gives entire functions $g$ and $h$ such that
\[
 f=g+\overline h,
 \qquad h(0)=0.
\]
Applying the identity \eqref{eq:circle-ball-mean} to $v=gh$, and noting $(gh)(0)=0$, we obtain
\[
 \int_{B(0,R)}g(z)h(z)e^{-2\phi(z)}\dV(z)=0.
\]
Expanding $\abs{g+\overline h}^2$ now yields
\[
 \int_{B(0,R)}\abs{f(z)}^2e^{-2\phi(z)}\dV(z)
 =\int_{B(0,R)} \bigl(\abs{g(z)}^2+\abs{h(z)}^2\bigr)e^{-2\phi(z)}\dV(z).
\]
Letting $R\to\infty$ and applying monotone convergence, we find
\begin{equation}\label{eq:circle-part-norms}
 \norm{g}_\phi^2+\norm{h}_\phi^2
 =\norm{f}_\phi^2<\infty.
\end{equation}
In particular, $g,h\in\F_\phi^2$.

If $\mathcal C_\phi=\{0\}$, then $h\in\F_{\phi,\circ}^2$. Otherwise $1\in\F_\phi^2$, and $h$ is integrable against $e^{-2\phi}dV$ by Cauchy--Schwarz. Applying~\eqref{eq:circle-ball-mean} to $h$ and letting $R\to\infty$, we obtain $\inner{h}{1}_\phi=0$. Thus $h\in\F_{\phi,\circ}^2$ in either case.

Every function in the inherited space therefore has its unique direct sum representation, and~\eqref{eq:circle-part-norms} shows that $S_\phi$ is unitary. The two spaces and Toeplitz forms agree, so Theorem~\ref{thm:generalized-direct-sum} applies with the same constants.
\end{proof}

\begin{corollary}[Comparable circle-invariant weights]
\label{cor:comparable-circle-weight}
Suppose that $\phi$ satisfies \eqref{eq:generalized-weight}. Assume that there is another weight $\psi$ satisfying \eqref{eq:generalized-weight} and \eqref{eq:circle-invariant-weight}, and constants $0<m\leq M<\infty$, such that
\begin{equation}\label{eq:comparable-generalized-weights}
 m e^{-2\psi(z)}
 \leq e^{-2\phi(z)}
 \leq M e^{-2\psi(z)},
 \qquad z\in\C^n.
\end{equation}
Then $S_\phi$ is a Hilbert space isomorphism. The local mass Schatten criterion therefore holds on $\PH_{\phi,L^2}^2$ for every $0<p<\infty$.
\end{corollary}

\begin{proof}
The two weighted $L^2$ spaces have the same elements and equivalent norms. The same is true of their holomorphic and inherited pluriharmonic subspaces. In particular, $\mathcal C_\phi=\mathcal C_\psi$. Proposition~\ref{prop:circle-invariant-weight} gives a bounded decomposition $f=g+\overline h$ in the $\psi$ norm, with $h\perp_\psi\mathcal C_\psi$. We let $P_{\phi,\mathcal C}$ be the $\phi$ orthogonal projection onto $\mathcal C_\phi$, and we put
\[
 c=P_{\phi,\mathcal C}h,
 \qquad \widetilde h=h-c,
 \qquad \widetilde g=g+\overline c;
\]
then $f=\widetilde g+\overline{\widetilde h}$ and $\widetilde h\in\F_{\phi,\circ}^2$. Equivalent norms and boundedness of this finite-dimensional projection control $\widetilde g$ and $\widetilde h$ by $f$, which shows that $S_\phi$ is surjective. Its inverse is bounded by the open mapping theorem, and Proposition~\ref{prop:synthesis-transfer} applies.
\end{proof}

Condition \eqref{eq:circle-invariant-weight} is scalar circle invariance (rather than radiality); if $n\geq2$, every positive definite Hermitian form
\[
 \psi(z)=z^*Az
\]
has this invariance, and it is nonradial unless $A$ is a scalar matrix. Corollary~\ref{cor:comparable-circle-weight} also allows an arbitrary bounded nonradial perturbation of such a weight, provided the perturbed weight continues to satisfy \eqref{eq:generalized-weight}.

\subsection{Failure for the inherited norm}

We construct a weight in one complex dimension for which the holomorphic and antiholomorphic components nearly cancel in weighted $L^2$, but reinforce each other at selected points. Adding $\operatorname{Re}q$ to the Gaussian weight leaves its Levi form unchanged. We choose the entire function $q$ by polynomial approximation, prescribing one phase on sets carrying almost all of a translated Gaussian measure and a different phase at its center.

\begin{lemma}[Phase patching]\label{lem:phase-patching}
There are real numbers $a_j>0$ with $\abs{a_j-a_k}>4$ for $j\ne k$, an entire function $q$, and compact sets $K_j\subset\C$ such that, with
\[
 \varepsilon_j=2^{-5j},
 \qquad
 d\gamma_{a_j}(z)=\frac1\pi e^{-\abs{z-a_j}^2}\dV(z),
\]
the following properties hold:
\begin{align}
 \gamma_{a_j}(\C\setminus K_j)&<\varepsilon_j^2,
 \label{eq:patch-large-set}\\
 \sup_{z\in K_j}
 \abs{q(z)+a_jz-\frac{i\pi}{2}}&<\varepsilon_j,
 \label{eq:patch-phase}\\
 \abs{q(a_j)+a_j^2}&<\varepsilon_j.
 \label{eq:patch-center}
\end{align}
Each $K_j$ may be taken to be a finite union of pairwise disjoint closed disks.
\end{lemma}

\begin{proof}
Let $d\gamma_0(z)=\pi^{-1}e^{-\abs z^2}\dV(z)$. For each $j$, the Vitali covering theorem and inner regularity give a finite union $E_j$ of pairwise disjoint closed disks such that
\begin{equation}\label{eq:patch-model-set}
 0\notin E_j,
 \qquad
 \gamma_0(\C\setminus E_j)<\varepsilon_j^2.
\end{equation}
To obtain $E_j$, remove a disk about the origin and the complement of a large disk, with total $\gamma_0$ measure less than $\varepsilon_j^2/3$. We apply the Vitali theorem in the remaining annulus and choose finitely many disjoint disks whose union misses measure less than $\varepsilon_j^2/3$. Shrinking these disks slightly makes their closures disjoint and loses less than $\varepsilon_j^2/3$ of the measure, which proves \eqref{eq:patch-model-set}.

We choose $\rho_j>0$ so small that $\overline{B(0,\rho_j)}$ is disjoint from $E_j$, and put
\[
 \sigma_j=2^{-10j}.
\]
Then
\begin{equation}\label{eq:patch-errors}
 \sum_{k=j}^\infty\sigma_k<\varepsilon_j.
\end{equation}

We construct polynomials $q_j$ inductively. Begin with $q_0=0$ and $R_0>0$, and suppose that $q_{j-1}$ and $R_{j-1}$ have been chosen. Since $E_j$ is bounded, we may choose $a_j>0$ so large that
\[
 a_j+E_j,\qquad \overline{B(a_j,\rho_j)},\qquad
 \overline{B(0,R_{j-1})}
\]
are pairwise disjoint and $\abs{a_j-a_k}>4$ for $k<j$. The compact set
\[
 X_j=\overline{B(0,R_{j-1})}
 \cup(a_j+E_j)\cup\overline{B(a_j,\rho_j)}
\]
is a finite union of pairwise disjoint closed disks, so its complement is connected. Define $Q_j$ near the disjoint disks by
\[
 Q_j(z)=
 \begin{cases}
 q_{j-1}(z),&z\in\overline{B(0,R_{j-1})},\\
 -a_jz+i\pi/2,&z\in a_j+E_j,\\
 -a_j^2,&z\in\overline{B(a_j,\rho_j)}.
 \end{cases}
\]
Mergelyan's theorem gives a polynomial $q_j$ such that
\begin{equation}\label{eq:patch-runge-step}
 \sup_{X_j}\abs{q_j-Q_j}<\sigma_j.
\end{equation}
Choose $R_j>\max\{R_{j-1},j\}$ so that $X_j\subset B(0,R_j)$.

If $k>j$, then
\[
 \sup_{B(0,R_{k-1})}\abs{q_k-q_{k-1}}<\sigma_k.
\]
Thus $(q_j)$ converges uniformly on compact subsets to an entire function $q$. Set $K_j=a_j+E_j$. The change of variables $z=a_j+w$ gives
\[
 \gamma_{a_j}(\C\setminus K_j)
 =\gamma_0(\C\setminus E_j)<\varepsilon_j^2.
\]
Equation \eqref{eq:patch-runge-step} and a telescoping sum now give
\begin{align*}
 \sup_{z\in K_j}\abs{q(z)+a_jz-\frac{i\pi}{2}}
 &\leq\sum_{k=j}^\infty\sigma_k<\varepsilon_j,\\
 \abs{q(a_j)+a_j^2}
 &\leq\sum_{k=j}^\infty\sigma_k<\varepsilon_j.
\end{align*}
\end{proof}

\begin{theorem}[Failure of the inherited norm criterion]
\label{prop:inherited-generalized-failure}
There exist a real-valued function $\phi\in C^\infty(\C)$ satisfying
\begin{equation}\label{eq:failure-levi}
 dd^c\phi=\frac12dd^c\abs{z}^2
\end{equation}
and a finite positive discrete measure $\mu$ on $\C$ with the following properties:
\begin{enumerate}
\item For every $0<p<\infty$ and every $r>0$, the function $z\mapsto\mu(B(z,r))$ belongs to $L^p(\C,dV)$.
\item The holomorphic form is bounded on $\F_\phi^2$, and $T_{\mu,\phi}\in\Sp_p$ for every $0<p<\infty$.
\item The form \eqref{eq:generalized-toeplitz-form} is unbounded on $\PH_{\phi,L^2}^2$.
\end{enumerate}
\end{theorem}

\begin{proof}
Let $q$ and $(a_j)$ be supplied by Lemma~\ref{lem:phase-patching}, and set
\begin{equation}\label{eq:failure-weight}
 \phi(z)=\frac{\abs z^2}{2}+\operatorname{Re}q(z).
\end{equation}
Since $\operatorname{Re}q$ is harmonic, $\phi$ has the Gaussian Levi form~\eqref{eq:failure-levi} and satisfies~\eqref{eq:generalized-weight}.

For $a\in\mathbb R$, let
\[
 k_a(z)=\frac1{\sqrt\pi}\exp\left(az-\frac{a^2}{2}\right),
 \qquad
 g_a=e^qk_a,
 \qquad
 f_a=g_a+\overline{g_a}.
\]
The function $k_a$ is a normalized reproducing kernel for the classical Fock space with weight $e^{-\abs z^2}$ and
\begin{equation}\label{eq:failure-gauge-unitary}
 \mathcal U:\F_\phi^2\to\F_{\abs{z}^2/2}^2,
 \qquad \mathcal Ug=e^{-q}g,
\end{equation}
is unitary. Hence $g_a\in\F_\phi^2$ with $\norm{g_a}_\phi=1$, and $f_a\in\PH_{\phi,L^2}^2$.

With $z=x+iy$ and
\[
 \Theta_a(z):=\operatorname{Im}q(z)+ay,
\]
equation~\eqref{eq:failure-weight} gives
\begin{equation}\label{eq:failure-norm-identity}
 \norm{f_a}_\phi^2
 =4\int_\C\cos^2\Theta_a(z)\,d\gamma_a(z).
\end{equation}
At the real point $a$, the same calculation gives
\begin{equation}\label{eq:failure-value-identity}
 e^{-\phi(a)}\abs{f_a(a)}
 =\frac2{\sqrt\pi}\abs{\cos(\operatorname{Im}q(a))}.
\end{equation}

Put $f_j=f_{a_j}$. On $K_j$, equation \eqref{eq:patch-phase} gives
\[
 \abs{\Theta_{a_j}(z)-\pi/2}<\varepsilon_j,
 \qquad
 \cos^2\Theta_{a_j}(z)\leq\varepsilon_j^2.
\]
On the complement we use $\cos^2\Theta_{a_j}\leq1$; equations \eqref{eq:patch-large-set} and \eqref{eq:failure-norm-identity} then yield
\begin{equation}\label{eq:failure-small-norm}
 \norm{f_j}_\phi^2\leq8\varepsilon_j^2.
\end{equation}
On the other hand, equation~\eqref{eq:patch-center} implies $\abs{\operatorname{Im}q(a_j)}<\varepsilon_j<1$, so that \eqref{eq:failure-value-identity} gives
\begin{equation}\label{eq:failure-large-value}
 e^{-2\phi(a_j)}\abs{f_j(a_j)}^2\geq\frac1\pi.
\end{equation}

We now define
\begin{equation}\label{eq:failure-measure}
 \mu:=\sum_{j=1}^\infty2^{-j}\delta_{a_j}.
\end{equation}
The points $a_j$ are uniformly separated. For every fixed $r>0$, there is $N_r<\infty$ such that each ball of radius $r$ contains at most $N_r$ of these points. Hence, for every $p>0$,
\begin{align*}
 \int_\C\mu(B(z,r))^p\dV(z)
 &\leq N_r^{(p-1)_+}\sum_{j=1}^\infty2^{-jp}
 \int_\C\mathbf1_{B(a_j,r)}(z)\dV(z)\\
 &=N_r^{(p-1)_+}V(B(0,r))\sum_{j=1}^\infty2^{-jp}<\infty,
\end{align*}
which proves the first assertion.

The unitary map \eqref{eq:failure-gauge-unitary} also intertwines the holomorphic Toeplitz forms:
\[
 \int_\C g\overline h e^{-2\phi}\,d\mu
 =\int_\C(\mathcal Ug)\overline{(\mathcal Uh)}e^{-\abs z^2}\,d\mu.
\]
The classical local mass theorem therefore gives $T_{\mu,\phi}\in\Sp_p$ for every $p>0$.

Finally, let $F_j=f_j/\norm{f_j}_\phi$. Equations \eqref{eq:failure-small-norm} and \eqref{eq:failure-large-value} give
\begin{align*}
 \int_\C\abs{F_j(z)}^2e^{-2\phi(z)}\,d\mu(z)
 &\geq2^{-j}e^{-2\phi(a_j)}\abs{F_j(a_j)}^2\\
 &\geq\frac{2^{-j}}{8\pi\varepsilon_j^2}
 =\frac{2^{9j}}{8\pi}\to\infty.
\end{align*}
Since $\norm{F_j}_\phi=1$, the inherited form is unbounded.
\end{proof}

Equation \eqref{eq:patch-center} and the classical Fock pointwise estimate show that $e^{-q}\notin\F_{\abs{z}^2/2}^2$: $\abs{e^{-q(a_j)}}e^{-a_j^2/2}\to\infty.$ By~\eqref{eq:failure-gauge-unitary}, $1\notin\F_\phi^2$. Hence $\norm{f_j}_{\phi,\oplus}^2=2$, whereas \eqref{eq:failure-small-norm} shows that $\norm{S_\phi f_j}_\phi\to0$. Thus $S_\phi$ is not bounded below and does not have closed range. Multiplication by $e^{-q}$ gives a unitary equivalence for the holomorphic spaces, but does not preserve pluriharmonicity. The small inherited norms above show why that equivalence cannot transfer the criterion to $\PH_{\phi,L^2}^2$.

\section*{Acknowledgments}
The author acknowledges support from the Taussky--Todd Fellowship.

\end{document}